\documentclass[reqno]{amsart}

\usepackage{xypic,hyperref,todonotes,comment}
\usepackage{graphics,amssymb,multicol}
\usepackage{pagecolor,lipsum}
\usepackage{latexsym}
\usepackage{mathrsfs}
\xyoption{curve}

\definecolor{color1}{HTML}{6E001B}
\definecolor{color2}{HTML}{39E600}

\usepackage{hyperref}
\hypersetup{
colorlinks,
citecolor=color1,
filecolor=blue,
linkcolor=color1,
urlcolor=color1
}

\newcommand{\nocontentsline}[3]{}
\newcommand{\tocless}[2]{\bgroup\let\addcontentsline=\nocontentsline#1{#2}\egroup}

\newtheorem{lemma}{Lemma}[section]
\newtheorem{proposition}[lemma]{Proposition}
\newtheorem{theorem}[lemma]{Theorem}
\newtheorem{corollary}[lemma]{Corollary}

\newtheorem{predefinition}[lemma]{Proto-definition}

\newtheorem*{theoremA}{Theorem}

\theoremstyle{definition}

\newtheorem{definition}[lemma]{Definition}
\newtheorem{remark}[lemma]{Remark}

\newcommand{\mfk}[1]{\mathfrak{#1}}
\newcommand{\mbb}[1]{\mathbb{#1}}
\newcommand{\mcl}[1]{\mathcal{#1}}
\newcommand{\mrm}[1]{\mathrm{#1}}
\newcommand{\msc}[1]{\mathscr{#1}}

\newcommand{\msf}[1]{\mathsf{#1}}

\DeclareMathOperator{\Hom}{Hom}

\DeclareMathOperator{\RHom}{RHom}
\DeclareMathOperator{\REnd}{REnd}
\DeclareMathOperator{\Ext}{Ext}
\DeclareMathOperator{\Tor}{Tor}

\DeclareMathOperator{\Spec}{Spec}

\DeclareMathOperator{\Proj}{Proj}

\DeclareMathOperator{\Supp}{Supp}

\DeclareMathOperator{\supp}{supp}
\DeclareMathOperator{\Sym}{Sym}

\let\oldO\O

\newcommand{\Sing}{\operatorname{Sing}}
\newcommand{\sing}{\operatorname{sing}}
\newcommand{\stab}{\operatorname{stab}}
\newcommand{\Stab}{\operatorname{Stab}}

\newcommand{\ot}{\otimes}

\renewcommand{\O}{\mathscr{O}}
\renewcommand{\hat}{\widehat}
\renewcommand{\tilde}{\widetilde}
\renewcommand{\binom}[2]{{\Small\left(\begin{matrix}\ #1\ \\ #2 \end{matrix} \right)}}
\renewcommand{\b}[1]{[\!\hspace{.1mm}[{#1}]\!\hspace{.1mm}]}

\title[]{Hypersurface support for noncommutative complete intersections}
\date{\today}

\author{Cris Negron}
\address{Department of Mathematics, University of North Carolina, Chapel Hill, NC 27599}
\email{cnegron@email.unc.edu}

\author{Julia Pevtsova}
\address{Department of Mathematics, University of Washington, Seattle, WA 98195}
\email{julia@math.washington.edu}

\begin{document}
\maketitle

\begin{abstract}
We introduce an infinite variant of hypersurface support for finite-dimensional, noncommutative complete intersections.  By a noncommutative complete intersection we mean an algebra $R$ which admits a smooth deformation $Q\to R$ by a Noetherian algebra $Q$ which is of finite global dimension.  We show that hypersurface support defines a support theory for the big singularity category $\Sing(R)$, and that the support of an object in $\Sing(R)$ vanishes if and only if the object itself vanishes.  Our work is inspired by Avramov and Buchweitz' support theory for (commutative) local complete intersections.  In the companion piece \cite{negronpevtsova3}, we employ hypersurface support, and the results of the present paper, to classify thick ideals in stable categories for a number of families of finite-dimensional Hopf algebras.
\end{abstract}

\section{Introduction}

In continuing the studies of \cite{negronpevtsova}, we introduce a notion of \emph{hypersurface support} for infinite-dimensional modules over a ``noncommutative complete intersection".  By a noncommutative complete intersection we mean an algebra $R$ which admits a smooth deformation $Q\to R$ by a Noetherian algebra $Q$ which is of finite global dimension. In the sibling project \cite{negronpevtsova3}, we use hypersurface support for infinite-dimensional modules to classify thick ideals in certain tensor triangulated categories associated to finite-dimensional Hopf algebras.
\par

This work began as a generalization of Avramov and Buchweitz' theory of support for local (commutative) complete intersections, and the arguments employed in the text are often influenced by their commutative counterparts.  One can compare with \cite{avramovbuchweitz00, avramoviyengar18}, in particular.
\par

Let us now consider $k$ an arbitrary field, and $R$ a finite-dimensional algebra with prescribed smooth deformation $Q\to R$, by a Noetherian algebra $Q$ of finite global dimension.  To be clear, by a deformation we mean a choice of a flat $Z$-algebra $Q$ equipped with a map $Q\to R$ which reduces to an isomorphism $k\ot_ZQ\cong R$ at a distinguished point for $Z$, and by \emph{smoothness} we mean that the parametrizing algebra $Z$ is smooth (commutative and augmented) over $k$. We denote by $m_Z \subset Z$ the augmentation ideal of $Z$.
\par

As a basic example, one can consider a skew polynomial ring $Q=k_q[x_1,\dots, x_n]$, with skew commutator relations $x_ix_j-q^{a_{ij}}x_jx_i$ at $q$ a root of unity of order $l$, $Z$ the central algebra $k[x_1^l,\dots, x_n^l]$, and the truncation 
\[
R=\mbb{C}_q[x_1,\dots, x_n]/(x_1^l , \ldots, x_n^l)
\]
(cf.\ \cite{bensonerdmannholloway07,pevtsovawitherspoon09,pevtsovawitherspoon15}).  However, as noted in \cite[\S 2]{negronpevtsova}, there are various additional examples coming from studies in Hopf algebras.  For example, we have the small quantum group $u_q(\mfk{g})$ along with its deformation provided by the De Concini-Kac algebra $U^{DK}_q(\mfk{g})\to u_q(\mfk{g})$.
\par

Given such $Q\to R$ as above, one can consider the projective space $\mbb{P}(m_Z/m_Z^2)$ and assign to any point
\[
c:\Spec(K)\to \mbb{P}(m_Z/m_Z^2)
\]
a (noncommutative) hypersurface algebra $Q_c=Q_K/(f)$ where $f \mod m^2_{Z_K}$ is a representative for the point $c$.  For an $R$-module $M$, either finite-dimensional or infinite-dimensional, we say that $M$ is {\it supported} at such a point $c$ if the base change $M_K$ is of \emph{infinite} projective dimension over $Q_c$, and we define the hypersurface support of $M$ as
\[
\supp^{hyp}_\mbb{P}(M):=\left\{
\begin{array}{c}
\text{the image of all points }c:\Spec(K)\to \mbb{P}(m_Z/m_Z^2)\\
\text{at which  $\operatorname{projdim}_{Q_c} M_K = \infty$} 
\end{array}\right\}.
\]
To be clear, we restrict along the induced map $Q_c\to R_K$ to consider $M_K$ as a module over $Q_c$.  Also, the subscript $\mbb{P}$ in the notation $\supp^{hyp}_\mbb{P}$ indicates the projective space $\mbb{P}(m_Z/m_Z^2)$ specifically.  Our first result (in conjunction with Lemma \ref{lem:451} below) deals with the apparent ambiguity of this definition.
 
\begin{theoremA}[\ref{thm:AI}]
Consider $M$ any $R$-module, and $f,g\in m_Z$ with equivalent, nonzero, classes $\bar{f}=\bar{g}$ in $m_Z/m_Z^2$.  Then $M$ has finite projective dimension over $Q/(f)$ if and only if $M$ has finite projective dimension over $Q/(g)$. 
\end{theoremA}

To look at things from a different perspective, we can consider the singularity category $\Sing(R)$, which is the quotient $D^b(R)/\langle \Proj(R)\rangle$ of the bounded derived category by complexes of finite projective dimension.  An $R$-module $M$ is then supported at $c$ if its image along the exact map of triangulated categories $\Sing(R)\to \Sing(Q_c)$ (see Section \ref{sect:allmypeoplesisgorenstein}) is nonzero.
\par

From the singularity category perspective, it is clear that hypersurface support has the expected properties of a support theory for $\Sing(R)$ (see e.g.\ \cite{balmerfavi11}).  Namely, it is stable under shifting, splits over arbitrary sums $\supp^{hyp}_\mbb{P}(\oplus_\lambda M_\lambda)=\cup_\lambda M_\lambda$, and whenever there is an exact triangle $N\to M\to N'$, the support of $M$ lies in the union of the supports of $N$ and $N'$.  We show furthermore that the supports of finite-dimensional $R$-modules are in fact \emph{closed} subsets in $\mbb{P}(m_Z/m_Z^2)$, as a byproduct of the following result.

\begin{theoremA}[\ref{prop:493}]
For any finite-dimensional $R$-module $V$, there is an associated coherent sheaf $\mcl{E}_V$ on $\mbb{P}(m_Z/m_Z^2)$ such that
\[
\supp^{hyp}_\mbb{P}(V)=\Supp_\mbb{P}\mcl{E}_V.
\]
\end{theoremA}

The sheaf $\mcl{E}_V$ is constructed out of the extensions $\Ext_R^*(V, \Lambda)$ where $\Lambda$ is the maximal semisimple quotient of $R$. Hence, the right hand side of the equality in the above theorem is a cohomological support. In proving Theorem~\ref{prop:493} we also establish an identification between the hypersurface support of the present paper and the hypersurface support of \cite{negronpevtsova}, which is certainly expected.  

As a final result we prove a detection theorem for hypersurface support, which is an analog of Dade's Lemma \cite[Lemma 11.8]{dade78} in this context. 

\begin{theoremA}[\ref{thm:detection}, \ref{cor:g_dims}] 
For $M$ any $R$-module, $\supp^{hyp}_\mbb{P}(M)=\emptyset$ if and only if $M$ has finite projective dimension over $R$.
\end{theoremA}

The proof of the independence of representatives, Theorem \ref{thm:AI}, is modeled upon the proof of Avramov-Iyengar of the same result for (commutative) complete intersections \cite{avramoviyengar18}. The proof of the detection theorem, on the other hand, deviates from the proof of an analogous theorem in the commutative case: We apply our results on Koszul duality from \cite{negronpevtsova} whereas in \cite{bensoniyengarkrausepevtsova} the argument is inspired by \cite{bensoncarlsonrickard97} and uses the Kronecker quiver lemma. The Koszul duality approach yields a new proof of the ``infinite Dade's lemma" even in such a classical case as  a group algebra of an elementary abelian $p$-group over a field of characteristic $p$. 
\par

In the sibling text \cite{negronpevtsova3}, hypersurface support is used explicitly to classify thick ideals, and compute the spectrum of prime ideals, in the stable categories $\stab(\msf{u})$ for $\msf{u}$ among various classes of Hopf algebras.  We consider, for example, rings of functions $\O(\mcl{G})$ on a finite group scheme $\mcl{G}$, (bosonized) quantum complete intersections, Drinfeld doubles $\mcl{D}(B_{(1)})$ for Borel subgroups $B\subset \mbb{G}$ in almost-simple algebraic groups, and quantum Borels $u_q(B)$ in type $A$.

\subsection{Acknowledgements}

Thanks to Lucho Avramov, Eric Friedlander, Henning Krause, and Mark Walker for helpful commentary.  Special thanks to Srikanth Iyengar for offering myriad insights throughout the production of this work.  The first named author is supported by NSF grant DMS-2001608.  The second named author is supported by NSF grants DMS-1901854 and the Brian and Tiffinie Pang faculty fellowship.  This material is based upon work supported by the National Science Foundation under Grant No. DMS-1440140, while the first author was in residence at the Mathematical Sciences Research Institute in Berkeley, California, during the Spring 2020 semester, and the second author was in digital residence.

\tableofcontents

\section{Background and generic commentary}
\label{sect:global}

We work over a field $k$ of arbitrary characteristic.  All ``modules" are left modules.

\subsection{Smooth deformations}

Recall that a deformation of an algebra $R$ is a pair of a commutative augmented algebra $Z$ and an algebra map $Q\to R$ from a \emph{flat} $Z$-algebra $Q$ which reduces to an isomorphism $k\ot_ZQ\cong R$.  Note that, as part of this definition, we require that the structure map $Z\to Q$ has central image.
\par

By a \emph{(formally) smooth} deformation of an algebra $R$ we mean a deformation $Q\to R$ parameterized by a (formally) smooth algebra $Z$ over our base field $k$.  In the formal setting, we require that $Z$ is complete with finite-dimensional tangent space at its distinguished point.  This implies an isomorphism of local algebras $Z\cong k\b{y_1,\dots,y_n}$ (see \cite[\S3]{negronpevtsova} for a more detailed account).  Throughout we let $m_Z\subset Z$ denote the maximal ideal corresponding to the augmentation $1:Z\to k$.
\par

Note that a deformation parametrized by a smooth, i.e.\ finite type, algebra $Z$ can be completed to produce a formally smooth analog.  This is technically convenient, because working with local rings is convenient.

\subsection{The specific setup}
\label{sect:setup}

Throughout this work we consider specifically a deformation $Q\to R$ of a finite-dimensional algebra $R$ parametrized by a formally smooth algebra $Z$, which is isomorphic to a power series in finitely many variables.  We assume additionally that $Q$ is Noetherian and that the quotient $\Lambda=R/\operatorname{Jac}(R)$ is separable over $k$, so that all base changes $\Lambda_K$ along arbitrary field extensions $k\to K$ are also semisimple.  One should always consider the deformation $Q$ to be \emph{Gorenstein}--although this assumption is not strictly necessary for many results.
\par

We make a further assumption that the deformation $Q\to R$ admits an \emph{integral form}, by which we mean that $Q$ is obtained from a finite-type deformaton of $R$ via completion.  The relevance of this assumption does not appear until Section \ref{sect:hypersurfer}, at which point we explain the notion in more detail.

\subsection{New deformations from old}

Given a deformation $Q\to R$ as in \ref{sect:setup}, and $f\in m_Z$ with nonzero reduction $\bar{f}\in m_Z/m_Z^2$, we can produce another deformation $Q/(f)\to R$ parametrized by the algebra $Z/(f)$.  The algebra $Z/(f)$ will still be a power series ring, so that $Q/(f)\to R$ is still of the type demanded in Section \ref{sect:setup}.

\begin{remark}
Existence of an integral form for $Q/(f)$ is subtle, but can be dealt with.  (See Remark \ref{rem:944} below.) 
\end{remark}
\par

In the early sections of this text, when we consider ``a deformation $Q\to R$" we are considering either a deformation $Q\to R$ with $Q$ Noetherian and of finite global dimension, or some associated hypersurface $Q/(f)\to R$ thereof.  So we have these two flavors of deformation to consider: one of finite global dimension and one which is (generally) singular.
\par

At the point at which we begin to discuss hypersurface support specifically, we fix $Q$ to be of finite global dimension, and speak explicitly of the associated hypersurface deformations $Q/(f)\to R$.

\subsection{Smooth deformations and actions on categories}
\label{sect:sm00th}

Given a (formally) smooth deformation $Q\to R$, parametrized by a given (formally) smooth commutative algebra $Z$, we consider $A_Z$ the dg algebra
\[
A_Z:=\Sym\big(\Sigma^{-2}(m_Z/m_Z^2)^\ast\big),
\]
with vanishing differential.  Bezrukavnikov and Ginzburg \cite{bezrukavnikovginzburg07} show that the deformation $Q$ specifies an action of the algebra $A_Z$ on the derived category of $R$-modules
\begin{equation}\label{eq:204}
\iota_Q:A_Z\to Z\left(D(R)\right).
\end{equation}
This action lifts to an action on a certain dg category $D_{coh}(R)$ \cite[\S 3.4]{negronpevtsova}, and manifests concretely as a collection of compatible central algebra maps
\[
\iota_Q^M:A_Z\to \REnd_R(M),
\]
at arbitrary $M$.
\par

The algebra $A_Z$, along with its action on $D(R)$, can be seen as a generalization of the algebra of cohomological operators associated to a local complete intersection \cite{gulliksen74,eisenbud80,avramovbuchweitz00}.

\subsection{Hypersurface support: a preliminary report}

Consider $Q\to R$ a deformation as in Section \ref{sect:setup}, with $Q$ additionally of finite global dimension.  For a closed point $c\in \mbb{P}(m_Z/m_Z^2)$, we consider an associated hypersurface algebra $Q/(f)$, with $f\in m_Z$ such that $\bar{f}\in m_Z/m_Z^2$ is a (nonzero) representative for $c$.  We speak of $f\in m_Z$, specifically, as our representative for $c$.  When the choice of representative is irrelevant, we write $Q_c$ for an arbitrary expression $Q_c=Q/(f)$.  We similarly define hypersurface algebras $Q_c$ at non-closed points in $\mbb{P}(m_Z/m_Z^2)$ by employing base change (see Section \ref{sect:hypersurfer}).
\par

We would like to define the hypersurface support of $R$, relative to this deformation $Q$, as follows.

\begin{predefinition}
For $M$ an arbitrary $R$-module, we define the hypersurface support of $M$ as
\[
\supp^{hyp}_\mbb{P}(M)=\{c\in \mbb{P}(m_Z/m_Z^2): M_{k(c)}\ \text{\rm is of infinite projective dimension over }Q_c\}.
\]
\end{predefinition}

Here $M_{k(c)}$ is simply the base change $M\ot k(c)$ at the residue field for $c$.  We would then like to prove that this hypersurface support satisfies many desirable properties.  In pursuing this line of thinking, there are two primary issues which one has to deal with:
\begin{enumerate}
\item[(a)] One would like to show that this definition carries no ambiguity, in the sense that the projective dimension of $M_{k(c)}$ over $Q_c$ is independent of the choice of representative $f$ for $c$, and so independent of the specific choice of representing hypersurface algebra $Q_c=Q/(f)$.
\item[(b)] One would like to understand that the projective dimension for $M$ is encoded ``globally" on $\mbb{P}(m_Z/m_Z^2)$.
\end{enumerate}

For (b) we might mean, for example, that there is a sheaf $\mcl{E}_M$ on $\mbb{P}(m_Z/m_Z^2)$ whose fibers contain information about the projective dimension of $M$ at various $Q_c$.  In Proposition \ref{prop:dim_funs} and Theorem \ref{thm:AI} below, we deal with these fundamental issues (see also Proposition \ref{prop:KotExt}).  Having addressed these points, we formally define the hypersurface support in Section \ref{sect:hypersurfer}, and address some of its basic properties in Sections \ref{sect:detection}.

\section{Homological dimensions and derived functors}
\label{sect:dimensions}

For this section $Q\to R$ is a Noetherian deformation of a finite-dimensional algebra $R$ which is parametrized by a formally smooth algebra $Z\cong k\b{y_1,\dots,y_n}$, as in Section \ref{sect:setup}.  As usual $\Lambda=Q/\operatorname{Jac}(Q)=R/\operatorname{Jac}(R)$.  In our analysis of hypersurface support we will need the following basic result.

\begin{proposition}\label{prop:dim_funs}
Suppose that $M$ is any $\operatorname{Jac}(Q)$-torsion $Q$-module.  Then, for any given integer $d$, $M$ is of injective dimension $\leq d$ if and only if $\Ext^{>d}_Q(\Lambda,M)=0$.  Similarly, $M$ is of flat dimension $\leq d$ if and only if $\Tor^Q_{>d}(\Lambda,M)=0$.
\end{proposition}

By a torsion module we mean that each element in $M$ is annihilated by some power of the Jacobson radical.  In this case each Hom group $\Hom_Q(L,M)$ from a finitely generated module $L$ is $m_Z$-torsion, for $m_Z$ the maximal ideal in $Z$.  By Noetherianity of $Q$ the extensions $\Ext^\ast_Q(L,M)$ are $m_Z$-torsion as well.  Similarly, $\Tor_\ast^Q(L,M)$ is torsion for any finitely generated $L$. 
The proof of the Proposition is postponed till Section~\ref{sub:3.1}. 
\par

Of course, we are most interested in applying the above result in the case in which $M$ is an (infinite-dimensional) $R$-module.

\subsection{Application to Gorenstein deformations}

As remarked previously, we generally consider $Q$ to be a \emph{Gorenstein} deformation of $R$.  Recall that $Q$ is called Gorenstein (resp.\ $d$-Gorenstein) if $\operatorname{injdim}_Q(Q)<\infty$ (resp.\ $\operatorname{injdim}_Q(Q)\leq d$).  Specifically, $Q$ should be of finite injective dimension over itself both on the left and the right.  In this case, Proposition \ref{prop:dim_funs} proliferates as

\begin{corollary}\label{cor:g_dims}
Suppose that $Q$ is $d$-Gorenstein, and that $M$ is a $\operatorname{Jac}(Q)$-torsion module.  Then the following are equivalent:
\begin{multicols}{2}
\begin{enumerate}
\item[(a)] $\Ext^{\gg 0}_Q(\Lambda,M)=0$.\vspace{1mm}
\item[(c)] $\Tor^Q_{\gg 0}(\Lambda,M)=0$.\vspace{1mm}
\item[(e)] $\operatorname{injdim}(M)<\infty$.\vspace{1mm}
\item[(g)] $\operatorname{projdim}(M)<\infty$.\vspace{1mm}
\item[(i)] $\operatorname{flatdim}(M)<\infty$.\vspace{1mm}
\end{enumerate}
\columnbreak
\begin{enumerate}
\item[(b)] $\Ext^{>d}_Q(\Lambda,M)=0$.\vspace{1mm}
\item[(d)] $\Tor^Q_{>d}(\Lambda,M)=0$.\vspace{1mm}
\item[(f)] $\operatorname{injdim}(M)\leq d$.\vspace{1mm}
\item[(h)] $\operatorname{projdim}(M)\leq d$.\vspace{1mm}
\item[(j)] $\operatorname{flatdim}(M)\leq d$.\vspace{1mm}
\end{enumerate}
\end{multicols}
\end{corollary}

\begin{proof}
The equivalences between (e)--(j) are implied by general results of Iwanaga \cite[Theorem 2]{iwanaga80} \cite[Proposition 9.1.7]{enochsjenda00}.  Proposition \ref{prop:dim_funs} provides the equivalence between (a) and (e), and (c) and (i).  We therefor find (a)$\Leftrightarrow$(b) and (c)$\Leftrightarrow$(d) as well.
\end{proof}

\subsection{Primes and associated primes}

We recall that a prime ideal in $Q$ is a (two-sided) proper ideal $P$ in $Q$ such that the product of two ideals lies in $\mathfrak p$, $IJ\subset \mathfrak p$, if and only if $I$ is contained in $\mathfrak p$ or $J$ is contained in $\mathfrak p$.  By considering principal ideals, we see that primeness is equivalent to the implication
\[
aQb\subset \mathfrak p \ \Rightarrow \ a\in \mathfrak p\ \text{or}\ b\in \mathfrak p,\ \ \text{for any }a,b\in Q.
\]
From this latter description it is clear that prime ideals in commutative rings, in the above sense, are prime in the usual sense.  This expression of primeness also implies

\begin{lemma}
Consider $f:Z\to Q$ any map from a commutative algebra which has central image.  The preimage of a prime $\mathfrak p\subset Q$ along $f$ is a prime ideal in $Z$.
\end{lemma}

\begin{proof}
Consider $\mathfrak p$ a prime ideal in $Q$.  For central $a,b\in Q$, we have $aQb=abQ$, so that $aQb\subset \mathfrak p$ if and only if $ab\in \mathfrak p$.  Hence, $ab\in \mathfrak p$ implies $a\in \mathfrak p$ or $b\in \mathfrak p$.  Considering the case $a=f(x)$ and $b=f(y)$ for $x,y\in Z$, we see that $xy\in f^{-1}\mathfrak p$ implies $x\in f^{-1}\mathfrak p$ or $y\in f^{-1}\mathfrak p$.  Rather, $f^{-1}\mathfrak p$ is prime ideal in $Z$.
\end{proof}

One also has the notion of a prime module over $Q$.  This is a nonzero module $N$ over $Q$ such that any nonzero submodule $N'\subset N$ has $\operatorname{ann}(N')=\operatorname{ann}(N)$ (see e.g.\ \cite{goodearlwarfields04}).  One can see that the annihilator of a prime module is prime, and by considering the quotient $N=Q/\mathfrak p$ of $Q$ by a prime ideal, we see also that all prime ideals appear as annihilators of prime modules.

\begin{lemma}\label{lem:245}
Suppose that $N$ is a prime $Q$-module, and $x$ is a central element in $Q-\operatorname{ann}(N)$.  Then the multiplication operation $x:N\to N$ is injective.
\end{lemma}

\begin{proof}
The kernel $N'$ of $x\cdot-$ is a submodule in $N$, and hence is either $0$ or has annihilator equal to that of $N$.  The latter case cannot happen as $x\notin\operatorname{ann}(N)$, while $x\in \operatorname{ann}(N')$.  So we must have $N'=0$.
\end{proof}

\begin{proposition}[{\cite[Proposition 3.12]{goodearlwarfields04}}]
Any finitely generated, nonzero, module $M$ over $Q$ has a prime submodule.
\end{proposition}

By considering cyclic submodules of prime modules, and Noetherianity of $Q$, we have

\begin{corollary}\label{cor:259}
Any finitely generated $Q$-module $M$ admits a finite filtration $F_0M\subset F_1M\subset \dots\subset F_rM=M$ such that each subquotient $F_iM/F_{i-1}M$ is a cyclic prime $Q$-module
\end{corollary}

\subsection{Proof of Proposition \ref{prop:dim_funs}}
\label{sub:3.1}

The following proof is adapted from \cite[Proposition 5.5]{avramovfoxby91}.

\begin{proof}[Proof of Proposition \ref{prop:dim_funs}]
We prove the statement for $\Tor$ first.  We have
\[
\operatorname{flatdim}(M)=\operatorname{min}\{d:\Tor_{>d}^Q(M',M)=0\ \forall\ \text{fin gen'd }M'\},
\]
or $\operatorname{flatdim}(M)=\infty$ if no such $d$ exists \cite[Proposition 3.2.4]{weibel95}.  By Corollary \ref{cor:259} we find further that
\[
\operatorname{flatdim}(M)=\operatorname{min}\{d:\Tor^{>d}_Q(N,M)=0\ \forall\ \text{cyclic prime modules }N\},
\]
or $\operatorname{flatdim}(M)=\infty$ if no such $d$ exists.
\par

Now let's assume that $\Tor^{>d}_Q(\Lambda,M)=0$ as in the statement of Proposition and that $d$ is the minimal nonnegative integer with this property. 
We want to show that $\operatorname{flatdim}(M) =d$.  Consider now the possibly empty collection of prime ideals in $Q$,
\[
\msc{P}=\{\operatorname{ann}(N):N\text{ cyclic prime with }\Tor_{>d}^Q(N,M)\neq 0\}.
\]
We want to show that this collection is empty.  Suppose this is not the case, and choose $\mathfrak p$ a maximal element in $\msc{P}$, which must exist by Noetherianity of $Q$.  Let $N_\mathfrak p$ be its associated cyclic prime module.  Note that we have a surjection $Q/\mathfrak p\to N_\mathfrak p$, by cyclicity of $N_\mathfrak p$.  We claim that the maximal ideal $m_Z$ of $Z$ must be contained in $\mathfrak p$, in which case $Q/\mathfrak p$ and $N_\mathfrak p$ must be finite-dimensional.
\par

If $m_Z\nsubseteq \mathfrak p$, then chose $x\in m_Z-\mathfrak p$.  Applying Lemma \ref{lem:245}, we observe an exact sequence of modules
\begin{equation}\label{eq:285}
0\to N_\mathfrak p\overset{x}\to N_\mathfrak p\to N_\mathfrak p/xN_\mathfrak p\to 0
\end{equation}
Note that $\mathfrak p\subsetneq xQ+\mathfrak p\subset \operatorname{ann}(N_\mathfrak p/xN_\mathfrak p)$.  So now, by considering a filtrations of $N_\mathfrak p/xN_\mathfrak p$ as in Corollary \ref{cor:259} and maximality of $\mathfrak p$ in $\msc{P}$, we conclude that $\Tor_{>d}^Q(N_\mathfrak p/xN_\mathfrak p,M)=0$.  The long exact sequence on cohomology obtained from \eqref{eq:285} therefore implies that the action map
\sloppy{

}

\[
x\cdot-:\Tor_{>d}^Q(N_\mathfrak p,M)\to \Tor_{>d}^Q(N_\mathfrak p,M)
\]
is an isomorphism.  But this cannot happen as our torsion hypothesis implies that $\Tor_\ast^Q(N_\mathfrak p,M)$ is $m_Z$-torsion, and hence any $x\in m_Z$ annihilates elements in each nonzero Tor group.  So we conclude that $m_Z\subset \mathfrak p$, and that $N_\mathfrak p$ is finite-dimensional.  But now, $N_\mathfrak p$ is in the thick subcategory generated by $\Lambda$, so that
\[
\Tor_{>d}^Q(\Lambda,M)=0\ \Rightarrow\ \Tor_{>d}^Q(N_\mathfrak p,M)=0,
\]
which contradicts our assumption that $\Tor_{>d}^Q\in \msc{P}$, and hence contradicts our assumption that $\msc{P}$ is nonempty.  It follows that $\operatorname{flatdim}(M)=d$.

As for the injective dimension, by Baer's criterion we have
\[
\operatorname{injdim}(M)=\operatorname{min}\{d:\Ext^{>d}_Q(M',M)=0\ \forall\ \text{fin gen'd }M'\},
\]
or $\operatorname{injdim}(M)=\infty$ if no such $d$ exists.  So we may proceed exactly as above, with $\Tor$ replaced by $\Ext$, to obtain the desired result.
\end{proof}

\section{Gorenstein rings and big singularity categories}
\label{sect:gorenstein}

A Gorenstein (resp.\ $d$-Gorenstein) ring $R$ is a ring such that $\operatorname{injdim}_R(R)<\infty$ (resp.\ $\operatorname{injdim}_R(R)\leq d$).  For such a ring, objects of finite projective dimension and finite injective dimension coincide \cite{iwanaga80}.
\par

All of the rings which are of interest in this work are Gorenstein (see Lemma \ref{lem:gorenstein} below).    Singularity categories are employed in the formal definition of hypersurface support in Section \ref{sect:hypersurfer}. We now recall the construction of the singularity category for a Gorenstein ring.

\subsection{Singularity categories}

For $R$ a Gorenstein ring, we define the ``big" singularity category as
\[
\Sing(R):=D^b(R\text{-Mod})/\langle \Proj(R)\rangle,
\]
where $R$-Mod is the category of arbitrary $R$-modules.  This is the Verdier quotient of the bounded derived category by the thick subcategory of objects of fintie projective = finite injective dimension.  We also have the ``small" singularity category, which is the quotient of the derived category of finitely generated modules by the thick subcategory of perfect complexes
\[
\sing(R):=D^b(R\text{-mod})/\langle \operatorname{proj}(R)\rangle.
\]
Although we do not explicitly use the following result, it clarifies the manner in which the theory of the present paper is an extension of that of \cite{negronpevtsova}.

\begin{lemma}\label{lem:okok}
The category $\Sing(R)$ is compactly generated, and the functor $\sing(R)\to \Sing(R)$ is an equivalence onto the subcategory of compacts $\sing(R)\cong \Sing(R)^c$.
\end{lemma}

\begin{proof}
The map $R\text{-Mod}\to \Sing(R)$ restricted to the subcategory $\operatorname{GProj}(R)$ of Gorenstein projectives \cite{enochsjenda00} induces an equivalence from the stable category
\[
\underline{\operatorname{GProj}}(R)\overset{\cong}\longrightarrow \Sing(R).
\]
\cite[Theorem 5.6, Corollary 6.6]{iyengarkrause} \cite[Theorem 3.1]{berghoppermannjorgensen15}.  Similarly we have an equivalence between the stable category of finitely generated Gorenstein projectives and the small singularity category $\underline{\operatorname{Gproj}}(R)\cong \sing(R)$ \cite{buchweitz86}.  It is known now that $\underline{\operatorname{GProj}}(R)$ is compactly generated with compact objects identified with $\underline{\operatorname{Gproj}}(R)$, via the functor induced by the inclusion $\operatorname{Gproj}(R)\to \operatorname{GProj}(R)$ \cite[Proposition 2.10]{bensoniyengarkrausepevtsova20}.  So one considers the square connecting the two equivalences above to obtain the claimed result.
\end{proof}

As a more terrestrial comment, one can see by partially resolving bounded complexes that any object in $\Sing(R)$ is isomorphic to a shift $\Sigma^n M$ of some $R$-module $M$.  Furthermore, by considering syzygies and cosyzygies we see that the objects $R\text{-Mod}\subset \Sing(R)$ are stable under shifting.  So we see that the additive map
\[
R\text{-Mod}\to \Sing(R)
\]
is essentially surjective.  The following is well-known to experts and can be deduced, for example, by identifying the singularity category with the stable category of Gorenstein projectives and applying \cite[Theorem 10.2.14]{enochsjenda00}.

\begin{lemma}
The category $\Sing(R)$ admits arbitrary (set indexed) sums, and the functor $R\text{\rm-Mod}\to \Sing(R)$ commutes with sums.
\end{lemma}

\begin{remark}
In the presentation above we have used an identification between the singularity category of $R$ and its stable category of Gorenstein projectives.  One can also express the singularity category as the homotopy category of totally acyclic complexes of projectives (or injectives), as done in \cite{krause05} for example.
\end{remark}

\begin{remark}
In the case in which $R$ is Frobenius, the functor $R\text{-mod}\to \Sing(R)$ induces an equivalence from the big stable category $\Stab(R)\cong \Sing(R)$.
\end{remark}

\subsection{All of our algebras are Gorenstein}
\label{sect:allmypeoplesisgorenstein}

\begin{lemma}\label{lem:gorenstein}
Consider $Q\to R$ a deformation as in Section \ref{sect:setup}, and suppose that $Q$ is of finite global dimension.  Let $f_1,\dots, f_m$ be a regular sequence in the parametrizing algebra $Z$.  Then the quotient $Q/(f_1,\dots, f_m)$ is Gorenstein.  In particular, all hypersurface algebras $Q/(f)$ are Gorenstein, and $R$ is Gorenstein.
\end{lemma}

Our proof is adapted from \cite[Lemma 5.3]{cassidyconnerkirkmanmoore16}.

\begin{proof}
Take $I$ the ideal generated by the $f_i$ in $Q$, and take $d=\operatorname{gldim}(Q)$.  Via the Koszul resolution $Kos=Kos(f_1,\dots,f_m)$ over $Z$, and correposnding Koszul resolution $Kos_Q=Q\ot_Z Kos$ of $Q/I$ over $Q$, we see that $\RHom_Q(Q/I,Q)\cong \Sigma^{-m}Q/I$.  Hence
\[
\begin{array}{rl}
\RHom_{Q/I}(-,Q/I)& \cong \Sigma^m\RHom_{Q/I}(-,\RHom^\ast_Q(Q/I,Q))\\
& \cong \Sigma^m\RHom_Q(-,Q).
\end{array}
\]
This functor has cohomology vanishing in degrees $>d$, so that $\operatorname{injdim}(Q/I)\leq d$.
\end{proof}

Let us close the subsection with a remark about singularity categories.  By considering the Koszul resolution for $R$ over $Q$, we see that the restriction map $D^b(R)\to D^b(Q)$ sends the thick ideal $\langle \Proj(R)\rangle$ into $\langle \Proj(Q)\rangle$, and hence descends to a triangulated functor
\begin{equation}\label{eq:402}
\operatorname{res}:\Sing(R)\to \Sing(Q).
\end{equation}

\section{Hypersurface support}
\label{sect:hypersurfer}

We provide the promised definition of hypersurface support for arbitrary $R$-modules, when $R$ comes equipped with the appropriate deformation.  We first elaborate on the ``integral form" hypothesis of Section \ref{sect:setup}, then deal with some technical issues regarding base change in the formal setting.

\subsection{Integral forms}
\label{sect:setup2}

As mentioned in Section \ref{sect:setup}, we assume that our deformation $Q\to R$ admits an integral form $\mcl{Q}\to R$.  By this we mean that $R$ admits a deformation $\mcl{Q}\to R$ for which the parametrizing algebra $\mcl{Z}$ is smooth, and in particular of finite type over $k$, and such that the completion at the distinguished point $1:\mcl{Z}\to k$ recovers our original deformation.  That is to say, $Z=\hat{\mcl{Z}}_1$, the completion of $\mcl Z$ at the kernel of the augmentation map $1: \mcl Z \to k$,  and $Q\cong Z\ot_{\mcl{Z}}\mcl{Q}$.
\par

The point of this integral form is to pacify certain conflicts which arise when attempting to deal simultaneously with base change and completion.

\begin{remark}
One may ask at this point why we even work with complete local rings, rather than local rings which are essentially of finite type.  The reason is that usual localization does not behave well with Hopf structures.  For example, for $\mbb{G}$ an algebraic group, the local ring at the identity $\O_{\mbb{G},1}$ does not have a Hopf structure, while the complete local ring $\hat{\O}_{\mbb{G},1}$ is a Hopf algebra in the symmetric monoidal category of pro-finite vector spaces.
\end{remark}

\subsection{Topological base change}

Consider a deformation $Q\to R$ as in Section \ref{sect:setup}, with some integral form $\mcl{Q}\to R$.  For a field extension $k\to K$, we let $\mcl{Z}_K$, $\mcl{Q}_K$, and $R_K$ denote the usual base change $-_K=K\ot-$, so that we have the base changed deformation $\mcl{Q}_K\to R_K$ parametrized by $\mcl{Z}_K$.  We have the standard base change operation
\[
(-)_K:R\text{-Mod}\to R_K\text{-Mod},\ \ M\mapsto M_K=K\ot M.
\]
\par

We complete at the distinguished point $1:\mcl{Z}_K\to K$ to obtain the corresponding formal deformation $Q_K\to R_K$ parametrized by $Z_K=\hat{(\mcl{Z}_K)}_1$.  So, from another perspective, we obtain the formal deformation $Q_K\to R_K$ via ``topological base change" $Z_K=\varprojlim_n (Z/m_Z^n)\ot K$ and $Q_K=\varprojlim_n (Q/m_Z^nQ)\ot K$.  For $Z=k\b{y_1,\dots,y_n}$ we have $Z_K=K\b{y_1,\dots,y_n}$.

\begin{lemma}\label{lem:505}
The base change $Q_K$ is Noetherian and of finite global dimension.
\end{lemma}

\begin{proof}
$Q_K$ is Noetherian as it is finite over the Noetherian algebra $Z_K$.  We have a finite projective resolution of $P\to \Lambda$ of $\Lambda$ over the integral form $\mcl{Q}$, at least after localizing around $1\in\Spec(\mcl{Z})$.  Apply base change to obtain a finite resolution $P_K\to \Lambda_K$ over $\mcl{Q}_K$, then apply the exact functor $Z_K\ot_{\mcl{Z}_K}-$ to obtain a finite resolution $\hat{P}_K\to \Lambda_K$ over $Q_K$.  So, for $d$ the length of the resolution $P$, we have $\Tor_{>d}^{Q_K}(\Lambda_K,-)=0$.  By considering minimal resolutions for finitely generated $Q_K$-modules, this vanishing of $\Tor$ implies that $Q_K$ is of global dimension $\leq d$.
\end{proof}

\begin{lemma}\label{lem:base_change}
Consider a field extension $k\to K$.  For any finite-dimensional $R$-module $V$, and arbitrary $R$-module $M$, the natural maps
\[
K\ot\Tor^Q_\ast(V,M)\to \Tor^{Q_K}_\ast(V_K,M_K)\ \ \text{and}\ \ K\ot\Ext^\ast_Q(V,M)\to \Ext^\ast_{Q_K}(V_K,M_K)
\]
are isomorphisms.
\end{lemma}

\begin{proof}
As in the proof of Lemma \ref{lem:505}, one employs the integral form $\mcl{Q}$, notes that $\Tor$ and $\Ext$ over $\mcl{Q}$ agree with those of $Q$ (for $R$-modules), and recalls that such base change formulae are valid over $\mcl{Q}$.
\end{proof}

\subsection{Representative independence over hypersurfaces}
\label{sect:rep_indep}

We have the following infinite analog of \cite[Corollary 5.3]{negronpevtsova}.  The proof is taken almost directly from \cite{avramoviyengar18}, and so is delayed until the appendix.

\begin{theorem}[{cf.\ \cite[Theorem 2.1]{avramoviyengar18}}]\label{thm:AI}
Let $Q\to R$ be a deformation as in Section \ref{sect:setup}, and assume that $Q$ is of finite global dimension.  Consider $M$ any $R$-module, and $f,g\in m_Z$ with equivalent, nonzero, classes $\bar{f}=\bar{g}$ in $m_Z/m_Z^2$.  Then there is an identification of graded vector space
\[
\Tor_\ast^{Q/(f)}(\Lambda,M)=\Tor_\ast^{Q/(g)}(\Lambda,M).
\]
Similarly, there is an identification
\[
\Ext^\ast_{Q/(f)}(\Lambda,M)=\Ext^\ast_{Q/(g)}(\Lambda,M).
\]
In particular, these (co)homology groups vanish in high degree over $Q/(f)$ if and only if they vanish in high degree over $Q/(g)$.
\end{theorem}

Of course, the hypersurface algebra $Q/(f)$ depends on $f$ only up to nonzero scaling, $Q/(f)=Q/(\zeta f)$ for any $\zeta\in k^\times$.  So, to phrase things slightly differently, the claimed identification of (co)homology groups holds whenever the reductions $\bar{f}$ and $\bar{g}$ represent the same point in projective space $\mbb{P}(m_Z/m_Z^2)$.

\subsection{Hypersurface support}

Fix now $R$ a finite-dimensional algebra with prescribed deformation $Q\to R$ as in Section \ref{sect:setup}, and suppose that $Q$ is of {\bf finite global dimension}.
\par

Consider a point $c:\Spec(K)\to \mbb{P}(m_Z/m_Z^2)$, which we may view as a map into the base change $\Spec(K)\to \mbb{P}(m_Z/m_Z^2)_K=\mbb{P}(m_{Z_K}/m_{Z_K}^2)$.  We say a hypersurface algebra $Q_K/(f)$ is a \emph{hypersurface representative} for $c$ if the class of $f$ in $m_{Z_K}/m_{Z_K}^2$ lies in the punctured line defined by $c$.  As a corollary to Theorem \ref{thm:AI} we have

\begin{corollary}\label{cor:836}
Consider $Q_K/(f)$ and $Q_K/(g)$ two hypersurface representatives for a given point $c:\Spec(K)\to \mbb{P}(m_Z/m_Z^2)$.  Then for any $R$-module $M$, the base change $M_K$ is of finite projective dimension over $Q_K/(f)$ if and only if $M_K$ is of finite projective dimension over $Q_K/(g)$.
\end{corollary}

\begin{proof}
After rescaling if necessary, we may assume that $f$ and $g$ define the same class in $m_{Z_K}/m_{Z_K}^2$.  The result now follows by Corollary \ref{cor:g_dims} and Theorem \ref{thm:AI}.
\end{proof}

\begin{remark}\label{rem:944}
A subtle point here is that Corollary \ref{cor:836} allows us to always assume that $f\in m_Z$ lies in the integral form $\mcl{Z}$, when desired, so that the hypersurface deformation $Q/(f)$ can always be assumed to have an integral form $\mcl{Q}/(f)$.
\end{remark}

When the distinction is unimportant we take, for a given point $c:\Spec(K)\to \mbb{P}(m_Z/m_Z^2)$,
\[
Q_c=Q_K/(f)\ \text{for any representative $f$ of the point }c.
\]
By the above lemma, the kernel of the functor
\begin{equation}\label{eq:Fc}
F_c:\Sing(R)\to \Sing(Q_c),
\end{equation}
defined as base change $(-)_K:\Sing(R)\to \Sing(R_K)$ composed with restriction \eqref{eq:402}, is independent of the choice of representative for the point $c$.

\begin{lemma}\label{lem:451}
Consider any field extension $K\to K'$, and $f\in m_{Z_K}$ with non-trivial reduction $\bar{f}\in (m_Z/m_Z^2)_K$.  Then for any $R$-module $M$, $M_K$ is of finite projective dimension over $Q_K/(f)$ if and only if $M_{K'}$ is of finite projective dimension over $Q_{K'}/(f)$.
\end{lemma}

\begin{proof}
The result follows by the base change formula of Lemma \ref{lem:base_change} (Remark \ref{rem:944}) and Corollary \ref{cor:g_dims}.
\end{proof}

Lemma \ref{lem:451} says that if a given point $c':\Spec(K')\to \mbb{P}$ factors through some other point
\[
\Spec(K')\to \Spec(K)\overset{c}\to \mbb{P}(m_Z/m_Z^2),
\]
then the kernels of the two functors $F_c:\Sing(R)\to \Sing(Q_c)$ and $F_{c'}:\Sing(R)\to \Sing(Q_{c'})$ agree.

\begin{definition}\label{def:supp_hyper}
We say an $R$-module $M$ is supported at a given point $c:\Spec(K)\to \mbb{P}(m_Z/m_Z^2)$ if $M$ is \emph{not} in $\ker\{F_c: \Sing(R)\to \Sing(Q_c)\}$.  Equivalently, $M$ is supported at $c$ if $M_K$, considered as a $Q_c$-module via the deformation $Q_c\to R_K$, is of infinite projective dimension over $Q_c$.  We take
\[
\supp^{hyp}_\mbb{P}(M):=\left\{
\begin{array}{c}
\text{the image of all points }c:\Spec(K)\to \mbb{P}(m_Z/m_Z^2)\\
\text{at which }\operatorname{projdim}_{Q_c}(M_K)=\infty
\end{array}\right\}.
\]
\end{definition}

Corollary \ref{cor:836} and Lemma \ref{lem:451} tell us that there are no ambiguities in the above definition.  We note that the support $\supp^{hyp}_\mbb{P}(M)$ is a subset in $\mbb{P}(m_Z/m_Z^2)$ which is not necessarily closed.  We see below that $\supp^{hyp}_\mbb{P}(M)$ is closed whenever $M$ is finite-dimensional.

\subsection{Support for finite-dimensional modules}
\label{sect:fd}

Recall our deformation algebra $A_Z=\Sym(\Sigma^{-2}(m_Z/m_Z^2)^\ast)$, from Section \ref{sect:sm00th}, and its corresponding action on $D^b(R)$.
\par

We have $\Proj(A_Z)=\mbb{P}(m_Z/m_Z^2)$ and the natural $A_Z$-action on extensions provides, for any $R$-modules $N$ and $M$, an associated quasicoherent sheaf $\Ext^\ast_R(N,M)^\sim$ on $\mbb{P}(m_Z/m_Z^2)$.  By \cite[Corollary 4.7]{negronpevtsova} these extension sheaves are \emph{coherent} when $N$ and $M$ are finite-dimensional.

\begin{proposition}\label{prop:493}
For a finite-dimensional $R$-module $V$ over $R$ we have
\[
\supp^{hyp}_\mbb{P}(V)=\Supp_\mbb{P}\Ext^\ast_R(V,\Lambda)^\sim.
\]
In particular, the support of a finite-dimensional $R$-module is closed in $\mbb{P}(m_Z/m_Z^2)$.
\end{proposition}

In the proof we employ the following construction: To any point $c:\Spec(K)\to \mbb{P}(m_Z/m_Z^2)$ we have an associated graded algebra map $\phi_c:A_Z\to K[t]$ which is nonzero in degree $2$ and annihilates the kernel of $c:A_Z^2=(m_Z/m_Z^2)^\ast\to K$.  We also consider the base change $\phi_{c,K}:(A_Z)_K\to K[t]$, which is now a graded algebra surjection.  Note that although the map $c$ is really only defined up to scaling, its kernel is well-defined and uniquely determined by the point $c$.  Similarly, the map $\phi_c$ is uniquely determined up to the $K^\times$-action on $K[t]$ by graded automorphisms.

\begin{proof}
Consider an arbitrary point $c:\Spec(K)\to \mbb{P}$.  We want to show that $V$ is supported at $c$ if and only if the image $c(pt)\in \mbb{P}$ of $c$ is in the support of the graded $A_Z$-module $\Ext^\ast_R(V,\Lambda)$.  By Lemma \ref{lem:451} we are free to assume that $c$ is a geometric point, that is, that $K=\overline{K}$.  (We make this assumption to conform to the setting of \cite{negronpevtsova}.)
\par

By \cite[Lemma 6.8]{negronpevtsova}, we have $\Ext^{\gg 0}_{Q_K/(f)}(V_K,\Lambda_K)=0$ if and only if the fiber of
\[
\Ext^\ast_{R_K}(V_K,\Lambda_K)=K\ot\Ext^\ast_R(V,\Lambda)
\]
along the localized map $\phi_{c,K}:K\ot A_Z\to K[t,t^{-1}]$ vanishes.  So, $V$ is supported at the $K$-point $c$, in the sense of Definition \ref{def:supp_hyper}, if and only if the pullback $\pi^\ast\Ext^\ast_R(V,\Lambda)^\sim$ along the projection $\pi:\mbb{P}_K\to\mbb{P}$ is supported at $c$.  Since, for any coherent sheaf $\mcl{F}$ on $\mbb{P}$,
\[
\Supp(\pi^\ast\mcl{F})=\pi^{-1}\Supp(\mcl{F})
\]
\cite[Tag \href{https://stacks.math.columbia.edu/tag/056H}{056H}]{stacks}, we see that $V$ is supported at $c$ if and only if $\Ext^\ast_R(V,\Lambda)^\sim$ is supported at $c(pt)$, as desired. 
\end{proof}

\begin{remark}
By Proposition \ref{prop:493} and \cite[Corollary 6.10]{negronpevtsova} we see that the hypersurface support of the present paper agrees with that of its predecessor \cite[Definition 6.7]{negronpevtsova}.
\end{remark}

\section{The detection theorem}
\label{sect:detection}

Fix for this section $R$ is a finite-dimensional algebra which is equipped with a deformation $Q\to R$ as in Section \ref{sect:setup}, and we assume additionally that $Q$ is of {\bf finite global dimension}.  We define the support $\supp^{hyp}_\mbb{P}(M)$ of $R$-modules via this given deformation (Definition \ref{def:supp_hyper}). The point of this final section of the text is to prove the following detection theorem.

\begin{theorem}\label{thm:detection}
For $M$ any $R$-module, $\supp^{hyp}_\mbb{P}(M)=\emptyset$ if and only if $M$ has finite projective dimension over $R$.  That is to say, hypersurface support for $M$ vanishes if and only if $M$ vanishes in the singularity category.
\end{theorem}

The proof is essentially an application of (dg) commutative algebra.

\begin{remark}
In the case of finite-dimensional $M$, Theorem \ref{thm:detection} follows easily from the cohomological expression of Proposition \ref{prop:493}.
\end{remark}

\subsection{Dg modules and sheaves on $\mbb{P}$}

Fix $S=k[\xi_1,\dots, \xi_n]$ a polynomial ring in $n$ variables, and $\mbb{P}=\Proj(S)$.  We consider $S$ as a dg algebra, with $\deg(\xi_i)=2$ for all $i$ and vanishing differential.  Following standard notation we take $S_{(f)}$ the degree $0$ portion of the localization $S_f$, for homogenous $f$, and $N_{(f)}=(N_f)^{0}$ for any graded $S$-module $N$.  For a degree $2$ function $f\in S$, let $U_f\subset \mbb{P}$ denote the basic open $U_f=\Spec(S_{(f)})$.

For any dg $S$-module $N$ the differential $d:N\to N$ is a degree $1$ $S$-linear homomorphism, and each localization $N_f$ is a dg module over $S_f$.  We can consider each $N_f$ as a dg module over the (non-dg) ring $S_{(f)}$, in which case it becomes an unbounded complex of $S_{(f)}$-modules and we are in a classical (non-dg) situation.  The dg modules $N_f$ glue over the localizations $S_{(f)}$, $\deg(f)=2$, to produce a sheaf $\mcl{F}_N$ of quasi-coherent dg $\O_\mbb{P}$-modules on projective space $\mbb{P}=\Proj(S)$.  Indeed, as an $\O_\mbb{P}$-module $\mcl{F}_N$ is assembled from the usual sheaves associated to $N^{ev}$, $N^{odd}$, and their shifts.
\par

The above construction is natural in $N$, so that in total we obtain an exact functor from dg $S$-modules to dg sheaves on projective space.

\begin{definition}
We let $\mcl{F}_?:S\text{-dgmod}\to \operatorname{QCoh}_{dg}(\mbb{P})$ denote the exact functor constructed above, $N\mapsto \mcl{F}_N$.
\end{definition}

Consider a geometric point $c:\Spec(K)\to \mbb{P}$, which determines (up to $K^\times$-action) a graded surjection $\phi_c:S_K\to K[t]$, $\deg(t)=2$.  If this point lies in $U_f\subset \mbb{P}_K$, then $f\in S_K^2$ maps to a non-zero scaling of $t$, and we may generally assume that this scalar is $1$.

\begin{lemma}\label{lem:fibers}
Consider a geometric point $c:\Spec(K)\to \mbb{P}$, and a standard open $U_f\subset \mbb{P}_K$ containing $c$.  Then there is a natural isomorphism
\[
\mrm{L}c^\ast \mcl{F}_N\cong K[t,t^{-1}]\ot_{S_K}^{\rm L}N_K
\]
in $D(K)$, where we change base along the (localized) map $\phi_c:S_K\to K[t,t^{-1}]$.
\end{lemma}

\begin{proof}
By changing base initially we may assume $K=k$, so that $c$ is a closed point, and by replacing $N$ with a semi-projective resolution we may assume that $N$ is K-flat \cite[Definition 5.1]{spaltenstein88}. In this case the derived fiber and derived product are identified with their underived counterparts.  It suffices to provide an isomorphism $k\ot_{S_{(f)}}N_f\cong k[t,t^{-1}]\ot_{S}N$.  The latter space is isomorphic to $k[t,t^{-1}]\ot_{S_f}N_f$, since $f$ maps to a non-zero scaling of $t$, $\phi(f)=\zeta t$.
\par

We have the two maps
\[
k\ot_{S_{(f)}}N_f\to k[t,t^{-1}]\ot_{S_f}N_f,\ \ \xi\ot m\mapsto \xi\ot m
\]
and
\[
k[t,t^{-1}]\ot_{S_f}N_f\to k\ot_{S_{(f)}}N_f,\ \ \xi t^n\ot m\mapsto \xi\ot \zeta^{-n}f^nm.
\]
One observes directly that these maps are mutually inverse, and so provide the desired isomorphism.
\end{proof}

\begin{lemma}\label{lem:168}
At all $f\in S^2$, $H^\ast(\mcl{F}_N)|_{U_f}\cong H^\ast(N)_f$.
\end{lemma}

\begin{proof}
We have $H^\ast(\mcl{F}_N)|_{U_f}=H^\ast(\mcl{F}_N|_{U_f})=H^\ast(N_f)$.  By exactness of localization, and the fact that the differential $d:N\to N$ is $S$-linear, we localize the exact sequences
\[
0\to Z(N)\to N\overset{d}\to N,\ \ N\overset{d}\to N\to N/B(N)\to 0
\]
to get exact sequences 
\[
0\to Z(N)_f\to N_f\overset{d}\to N_f,\ \ N_f\overset{d}\to N_f\to N_f/B(N)_f\to 0.
\]
This gives $Z(N)_f\cong Z(N_f)$, $B(N)_f\cong B(N_f)$, and therefore $H^\ast(N_f)\cong H^\ast(N)_f$.
\end{proof}

\subsection{A vanishing theorem}

\begin{theorem}\label{thm:vanish}
Suppose that $N$ is a bounded below dg $S$-module, and that at each geometric point $c:S\to K[t]$ the fiber $K[t]\ot^{\rm L}_S N$ has cohomology vanishing in degrees $>d$, for some (uniformly chosen) integer $d$.  Then $N$ has bounded cohomology.
\end{theorem}

To be clear, by a geometric point we mean a graded algebra map $S\to K[t]$, with $\deg(t)=2$ and $K=\bar{K}$, such that the base change $S_K\to K[t]$ is surjective.

\begin{proof}
The proof is by induction on the number of generators of $S$.  Before proceeding with the induction argument, we first claim that for $N$ as in the statement of the theorem, all of the localizations of cohomology $H^\ast(N)_f$, $f\in S^2$, vanish.
\par

Indeed, boundedness of the fibers $K[t]\ot^{\rm L}_S N$ implies that all of the localizations $K[t,t^{-1}]\ot^{\rm L}_S N$ vanish, and by Lemma \ref{lem:fibers} we conclude that all of the derived fibers $\mrm{L}c^\ast \mcl{F}_N$ of the associated sheaf vanish, at all geometric points $c:\Spec(K)\to \mbb{P}$.  By Neeman \cite[Lemma 2.12]{neeman92}, this implies that $H^\ast(\mcl{F}_N)=0$, and by Lemma \ref{lem:168} we conclude $H^\ast(N)_f=0$ at all $f\in S^2$.
\par

Now, let us proceed with our induction argument.  Take $n=\dim(S^2)$, and consider for the base case the algebra $k[t]$, $\deg(t)=2$.  If the base change $\bar{k}[t]\ot_{k[t]}N=N_{\overline{k}}$ of any dg $k[t]$-module along the map $k[t]\to \overline{k}[t]$, $t\mapsto t$, has bounded cohomology then obviously $N$ has bounded cohomology, as desired.  Now suppose that the claimed result holds for all polynomial rings $S'$ in $<n$ variables.
\par

Consider $S'=S/(f)$ for any nonzero $f\in S^2$, and $N$ as prescribed.  Then the derived fibers of $S'\ot^{\rm L}_S N$ have uniformly bounded cohomology, and thus $S'\ot^{\rm L}_S N$ has bounded cohomology by our induction hypothesis.  By replacing $N$ with a semi-projective resolution, we may assume that $f$ acts as a nonzero divisor on $N$, and that $N/fN=S'\ot^{\rm L}_SN$.  We now have the exact sequence
\[
0\to N\overset{f}\to \Sigma^2 N\to \Sigma^2 N/\Sigma^2fN\to 0
\]
with $H^{>r-2}(N/fN)=H^{>r-2}(S\ot^{\rm L}_SN)=0$ at some large $r$.  By considering the long exact sequence on cohomology, the map
\[
f\cdot-:H^i(N)\to H^{i+2}(N)
\]
is therefore seen to be an isomorphism for all $i>r$.  It follows that $H^{>r}(N)$ either vanishes or has no $f$-torsion.  As the localization $H^\ast(N)_f$ vanishes, the latter case is impossible.  We conclude that $H^{\gg 0}(N)=0$.
\end{proof}

\subsection{Proof of Theorem \ref{thm:detection}}

As a final ingredient we have the following result, which was essentially proved in \cite[\S 6.3]{negronpevtsova}.  In the statement we employ the maps $\phi_c:A_Z\to K[t]$ discussed in Section \ref{sect:fd}.

\begin{proposition}\label{prop:KotExt}
At any geometric point $c:\Spec(K)\to \mbb{P}(m_Z/m_Z^2)$, with corresponding map $\phi_c:A_Z\to K[t]$ from $A_Z$, we have
\[
K[t]\ot^{\rm L}_{A_Z}\RHom_R(\Lambda,M)\cong \RHom_{Q_c}(\Lambda_K,M_K).
\]
\end{proposition}

\begin{proof}
Take $\phi_{c,K}:(A_Z)_K\to K[t]$ the base change of $\phi_c$, $A=A_Z$, and $A_c$ the corresponding algebra for the hypersurface deformation $Q_c\to R_K$. We have
\[
\RHom_{R_K}(\Lambda_K,M_K)\cong A_c\ot^t\RHom_{Q_c}(\Lambda_K,M_K)
\]
in the derived category of dg $A_c$-modules, with $A_c\ot^t\RHom_{Q_c}(\Lambda_K,M_K)$ a K-flat dg module \cite[Lemmas 4.1 \& 4.2]{negronpevtsova}.  Recall that $A_c^2=\ker(\phi_{c,K}^2)$.  Thus
\[
\begin{array}{l}
K[t]\ot^{\rm L}_{A}\RHom_Q(\Lambda,M)\cong K[t]\ot^{\rm L}_{A_K}K\ot\RHom_R(\Lambda,M)\vspace{1mm}\\
\hspace{2cm} \cong K[t]\ot^{\rm L}_{A_K}\RHom_{R_K}(\Lambda_K,M_K)\vspace{1mm}\\
\hspace{2cm} \cong K\ot^{\rm L}_{A_c}\RHom_{R_K}(\Lambda_K,M_K)\vspace{1mm}\\
\hspace{2cm} \cong K\ot_{A_c} (A_c\ot^t\RHom_{Q_c}(\Lambda_K,M_K))\cong \RHom_{Q_c}(\Lambda_K,M_K).
\end{array}
\]  
\end{proof}

We now prove our theorem.

\begin{proof}[Proof of Theorem \ref{thm:detection}]
Fix $d=\operatorname{gldim}(Q)$.  Recall that $\Ext^{\gg 0}_{Q_c}(\Lambda_K,M_K)=0$ if and only if $\Ext^{>d}_{Q_c}(\Lambda_K,M_K)=0$, by Corollary \ref{cor:g_dims}.  We have the formula
\[
K[t]\ot^{\rm L}_{A_Z}\RHom_R(\Lambda,M)\cong\RHom_{Q_c}(\Lambda_K,M_K)
\]
at arbitrary points $c:\Spec(K)\to \mbb{P}$.  So we apply Theorem \ref{thm:vanish}, with $S=A_Z$ and $N$ the dg $A_Z$-module $\RHom_R(\Lambda,M)$, and Corollary \ref{cor:g_dims}, to find that the hypersurface support for $M$ vanishes if and only if $H^\ast(\RHom_R(\Lambda,M))=\Ext^\ast_R(\Lambda,M)$ vanishes in high degree.  Such vanishing of cohomology occurs if and only if $M$ is of finite projective dimension over $R$, by Corollary \ref{cor:g_dims}.  So we have the desired relation, $\supp^{hyp}_\mbb{P}(M)=\emptyset$ if and only if $M$ vanishes in $\Sing(R)$.
\end{proof}

\appendix

\section{Proof of Theorem \ref{thm:AI}}
\label{sect:rep_indep}

We prove the representative independence result of Theorem \ref{thm:AI}.  Fix a deformation $Q\to R$ as in Section \ref{sect:setup}, with $Q$ of finite global dimension.  Take as usual $\Lambda=R/\operatorname{Jac}(R)$.

\subsection{A dg resolution of the simples}

We recall the Koszul resolutions $\mcl{K}_Z\overset{\sim}\to k$ of $k$ over $Z$, and corresponding resolution $\mcl{K}_Q\overset{\sim}\to R$ of $R$ over $Q$ obtained via base change \cite{bezrukavnikovginzburg07}.  In coordinates, $Z\cong k\b{y_1,\dots,y_n}$, the Koszul resolution is explicitly the dg $Z$-algebra $\mcl{K}_Z=\wedge^\ast_Z(\oplus_i Z\tilde{y}_i)$ with the $\deg(\tilde{y}_i)=-1$ and differential $d(\tilde{y}_i)=y_i\in Z$.  The dg algebra $\mcl{K}_Q$ is then obtained via base change $\mcl{K}_Q=Q\ot_Z \mcl{K}_Z$, and flatness of $Q$ over $Z$ implies that the induced map $\mcl K_Q \to R=Q\ot_Z k$ remains a quasi-isomorphism.

\begin{lemma}
There is a dg algebra resolution $T_Q\overset{\sim}\to \Lambda$ of the simples over $Q$ such that the following hold:
\begin{enumerate}
\item $T_Q$ comes equipped with a dg algebra inclusion $\mcl{K}_Q\to T_Q$, which restricts to a central inclusion $\mcl{K}_Z\to T_Q$, and an isomorphism in degree zero $Q\cong T^0_Q$.
\item $T_Q$ is bounded, non-positively graded, and finite and projective over $Q$ in each degree.
\end{enumerate}
\end{lemma}

\begin{proof}
Below, for an algebra $T$ and a $T$-bimodule $M$, we let $Tens_T(M)$ denote the tensor algebra, over $T$, generated by the bimodule $M$.  So, as a vector space $Tens_{T}(M)=\dots\oplus (M\ot_T M)\oplus M\oplus \mcl{K}$.  The product is given by tensor concatenation.  We construct $T_Q$ in a series of step, producing a tower of non-positively graded dg-algebras $\mcl{K}_Q = T_0 \to T_1 \to T_1 \to \dots$ such that $T_i$ is quasi-isomorphic to $\Lambda$ is degrees $> i$. The construction follows \cite{ciocankapranov01}, which in turn  mimics the classical construction of the universal covering space in topology. As a last step of the construction, we truncate the tower utilizing the finiteness of the global dimension of $Q$.
\par

Take $T_0=\mcl{K}_Q$ and consider degree $0$ cocycles $v^0_1,\dots, v^0_{n_0}$ which generate the kernel of the augmentation $Q=\mcl{K}_Q^0\to \Lambda$, modulo boundaries $B^0(\mcl{K}_Q)$.  Consider $M_1'$ the free dg $T_0$-bimodule with generators $\tilde{v}^0_i$ for each $v_i$ along with the dg bimodule map $M_1'\to T_0$, $\tilde{v}_i\mapsto v_i$.  More specifically, we take $M_1'$ the $\mcl{K}_Z$-central bimodule
\[
M'_1=\oplus_{i=1}^{n_0}(T_0\ot_{\mcl{K}_Z}T_0),\ \text{with degree $0$ generators }\tilde{v}_i^0,
\]
along with the proposed dg bimodule map $\xi_1:M'_1\to T_0$, $\xi_1(\tilde{v}_i^0)\mapsto v_i^0$.  Take $M_1$ the shifted dg bimodule $\Sigma M'_1$, with generators now in degree $-1$, and consider the dg $\mcl{K}_Z$-algebra
\[
T_1=Tens_{T_0}(M_1)\ \ \text{with differential}\ d_{T_1}|_{M_1^{\ot n}}=d_{M_1^{\ot n}}+\sum_{l=0}^{n-1} id^{\ot l}\ot \xi_1\ot id^{\ot n-l-1}.
\]
\par

The operation $d_{T_1}$ is a graded algebra derivation by construction, and we claim that it is in fact square $0$.  Since $d_{T_1}$ is a derivation, it suffices to check the formula $d_{T_1}^2=0$ on the generators $M_1\oplus T_0$.  But $d_{T_1}|_{M_1\oplus T_0}$ is just the differential on the mapping cone, $\operatorname{cone}(\xi_1)=(M_1\oplus T_0, d_{M_1}+d_{T_0}+\xi_1)$, so that the formula $d^2_{T_1}=0$ is verified.
\par

We have now a dg algebra $T_1$ satisfying (1) such that $H^0(T_1)=\Lambda$ and which is finite and projective over $Q$ in each degree.  We repeat the above process to produce a dg algebra $T_2=Tens_{T_1}(M_2)$ which is $\mcl{K}_Z$-central, finite and projective over $Q$ in each degree, and which comes equipped with a projection $T_2\to \Lambda$ which is a quasi-isomorphism in degrees $>-2$.  As the notation suggests, $T_2$ is produced by attaching a dg $T_1$-bimodule $M_2\to T_1$, with $M_2$ in degree $\leq -2$, which annihilates cohomology in degree $-1$.  Note that such $T_2$ comes equipped with an inclusion $T_1\to T_2$ of dg algebras which is an equality in degrees $>-2$, and that the augmentation $T_1\to \Lambda$ extends to a $\Lambda$-augmentation on $T_2$ which (necessarily) annihilates $M_2$.  We proceed recursively to produce such dg algebras and inclusions
\[
K_Q=T_0\to T_1\to \dots \to T_{N+1},
\]
up to $N=\max\{\operatorname{gldim}(Q),\operatorname{gldim}(Z)\}$.  To be clear, we now have a non-positively graded dg $\mcl{K}_Z$-algebra $T_{N+1}$ with an inclusion $\mcl{K}_Q\to T_{N+1}$ which is an equality in degree $0$ and a projection $T_{N+1}\to \Lambda$ which is a quasi-isomorphism in degrees $\geq -N$.  We have also that $T_{N+1}$ is finite and projective over $Q$ in each degree.
\par

We take finally
\[
T_Q=T_{N+1}/(T^{< -N}_{N+1}+B^{-N}T_N).
\]
Since $Q$ is of global dimension $\leq N$, $T_Q$ is projective over $Q$ in degree $-N$, and by the properties of $T_{N+1}$ the induced projection $T_Q\to \Lambda$ is in fact a quasi-isomorphism.  Furthermore, since $\operatorname{gldim}(Z)\leq N$, and $K_Q$ is concentrated in non-positive degrees $\geq -\operatorname{gldim}(Z)$, the injective map $K_Q\to T_{N+1}$ composed with the projection $T_{N+1}\to T_Q$ remains injective.
\end{proof}

\subsection{Systems of divided powers}

We recall some information, almost word for word, from \cite[Section 1.2]{avramoviyengar18}.  For $\msc{T}$ a dg algebra, and an integer $d$, we call a collection of elements $\{w^{(i)}\in \msc{T}^{2di}:i\geq 0\}$ a system of divided powers for $w=w^{(1)}$ if all of the $w^{(i)}$ are central, $w^{(0)}=1$,
\[
w^{(i)}w^{(j)}=\binom{i+j}{i}w^{(i+j)},\ \ \text{and}\ \ d(w^{(i)})=d(w)w^{(i-1)}.
\]
In characteristic $0$ any central element $w$ of even degree admits a unique system of divided powers, given by $w^{(i)}=(i!)^{-1}w^i$.  In positive characteristic this is not always the case.
\par

Given an odd degree central cocycle $z\in \msc{T}$, we let 
\[
\msc{T}\langle y:d(y)=z\rangle
\]
denote the divided power algebra $\oplus_{\geq 0} \msc{T} y^{(i)}$ with each of the $y^{(i)}$ central and extended differential $d(y^{(i)})=zy^{(i-1)}$.  As the notation suggests, the $y^{(i)}$ form a system of divided powers in $\msc{T}\langle y:d(y)=z\rangle$.

\subsection{The Tate construction for noncommutative hypersurfaces}
We provide a version of Tate's construction \cite{tate57} which is applicable to our noncommutative setting.  Our presentation is adapted directly from \cite[Section 1.4 \& 2.2]{avramoviyengar18}.
\par

Take $f\in m_Z$ and $z_f\in K_{Z}^{-1}\subset T_Q^{-1}$ any degree $-1$ element bounding $f$.  Take $\overline{T}_Q=Z/(f)\ot_Z T_Q$ and
\[
\msc{T}_f:=\overline{T}_Q\langle y: d(y)=z_f\rangle.
\]
Note that in the quotient $\overline{T}_Q=Z/(f)\ot_Z T_Q$ the degree $-1$ element $z_f$ is in fact a (central) cocycle, so that the definition of $\msc{T}_f$ makes sense.  The projection $T_Q\to \Lambda$ induces a projection
\[
\pi_f:\msc{T}_f\to \Lambda
\]
which (necessarily) annihilates all the $y^{(i)}$ for $i>0$.  This map is a map of dg $\mcl{K}_Z$-algebras, and of dg $Q/(f)$-bimodules.

\begin{lemma}\label{lem:193}
The cohomology $H^{-1}(\overline{T}_Q)$ is generated by the class of $z_f$, as a $Q$-module on the left or the right.
\end{lemma}

\begin{proof}
Since $f\in Q$ is central we have that $Qf$ is the ideal in $Q$ generated by $f$.  By a direct analysis of the diagram
\[
{\small
\xymatrix{
 & & Q z_f\ar[r]^{\rm surj}\ar[d]_{\rm incl}& Qf\ar[d]\\
0\ar[r] & B^{-1}(T_Q) \ar[r] & T^{-1}_Q \ar[d]\ar[r]^d & Q\ar[d]\\
 &  & \overline{T}_Q^{-1}\ar[r]^{\bar{d}} & Q/(f)
}
}
\]
we see that $\ker(\bar{d}^{-1})$ is generated by the image of the submodule $B^{-1}(T_Q)+Qz_f$ in $T^{-1}_Q$.  Consequently the map $Q z_f\to H^{-1}(\overline{T}_Q)$ is a surjective left module map.  We similarly find that the map $z_fQ\to H^{-1}(\overline{T}_Q)$ is surjective.
\end{proof}

\begin{lemma}[\cite{avramoviyengar18}]
The dg map $\pi_f:\msc{T}_f\to \Lambda$ is a quasi-isomorphism of dg $\mcl{K}_Z$-algebras, and dg $Q/(f)$-bimodules.
\end{lemma}

\begin{proof}
The proof is the same as \cite[Lemma 2.3]{avramoviyengar18}.  We have
\[
H^\ast(\overline{T}_Q)=\Tor^\ast_Q(\Lambda,Q/(f))=H^\ast(0\to\Lambda\ot_Q Q\overset{\Lambda\ot f}\to \Lambda\ot_Q Q\to 0)=\Sigma\Lambda\oplus \Lambda
\]
and the degree $-1$ cohomology is generated by the class of $z_f$, by Lemma \ref{lem:193}.  The result now follows by Lemma \ref{lem:193} and a straightforward spectral sequence argument \cite[Lemma 1.1(2)]{avramoviyengar18}.
\end{proof}

We note that $\msc{T}_f$ is finite an projective over $Q/(f)$ in each degree, so that the quasi-isomorphism $\pi_f:\msc{T}_f\overset{\sim}\to \Lambda$ is in particular a projective resolution of $\Lambda$ over $Q/(f)$.

\subsection{Representative independence}
\label{sect:AI_proof}

\begin{proof}[Proof of Theorem \ref{thm:AI}]
Take $\overline{\msc{T}}_f=k\ot_Z \msc{T}_f$ and $\overline{\msc{T}}_g=k\ot_Z \msc{T}_g$.  Explicitly,
\[
\overline{\msc{T}}_f=k\ot_Z\msc{T}_f=k\ot_ZT_Q\langle y:d(y)=z_f\rangle
\]
and we have a similar expression for $\overline{\msc{T}}_g$.  By \cite[Lemma 2.2]{avramoviyengar18} the (class of the) difference $z_f-z_g$ is bounded by an element $w$ in the reduction $k\ot_Z \mcl{K}_Z$.  Since the differential on this complex vanishes, we see that in fact the images of $z_f$ and $z_g$ agree in $k\ot_Z\mcl{K}_Z$ and hence in $k\ot_Z T_Q$.  So actually $\overline{\msc{T}}_f=\overline{\msc{T}}_g$ as dg $\mcl{K}_Z$-algebras and dg $R$-bimodules.  Now, we have
\[
\Lambda\ot^{\rm L}_{Q/(f)}M=\msc{T}_f\ot_{Q/(f)}M=\overline{\msc{T}}_f\ot_R M
\]
and similarly $\Lambda\ot^{\rm L}_{Q/(g)}M=\overline{\msc{T}}_g\ot_RM$, so that $\Lambda\ot^{\rm L}_{Q/(f)}M=\Lambda\ot^{\rm L}_{Q/(g)}M$.  Taking homology provides the claimed identification of graded $\Tor$ groups.
\par

The result for $\Ext$ is completely similar, as one calculates
\[
\RHom_{Q/(f)}(\Lambda,M)=\Hom_R\left(\overline{\msc{T}}_f,M\right)=\Hom_R\left(\overline{\msc{T}}_g,M\right)=\RHom_{Q/(g)}(\Lambda,M).
\]
\end{proof}

\renewcommand{\O}{\oldO}
\bibliographystyle{abbrv}

\end{document}